\documentclass[article,a4paper,oneside,11pt]{memoir}

\counterwithout{section}{chapter}

\usepackage[T1]{fontenc}
\usepackage[utf8]{inputenc}
\usepackage{amsmath,amssymb,amsthm}
\usepackage[hidelinks]{hyperref}
\usepackage{yfonts}
\usepackage{enumerate}
\usepackage{color}
\usepackage{mathtools}
\usepackage[ruled]{algorithm2e}
\usepackage{thmtools,thm-restate}
\usepackage{tikz}
\usepackage[justification=centering]{caption}

\usepackage{geometry}
\newcommand{\NN}{\mathbb{N}}

\renewcommand{\epsilon}{\varepsilon}

\newcommand\numberthis{\addtocounter{equation}{1}\tag{\theequation}}

\newtheorem{theorem}{Theorem}[section]
\newtheorem{lemma}[theorem]{Lemma}
\newtheorem{proposition}[theorem]{Proposition}
\newtheorem{corollary}[theorem]{Corollary}
\newtheorem{definition}[theorem]{Definition}
\newtheorem{claim}[theorem]{Claim}
\newtheorem{fact}[theorem]{Fact}
\newtheorem{question}[theorem]{Question}

\newcommand{\osref}[2]{%
  \setlength\abovedisplayskip{5pt plus 2pt minus 2pt}
  \setlength\abovedisplayshortskip{5pt plus 2pt minus 2pt}
  \ensuremath{\overset{\text{#1}}{#2}}
}

\title{$K_r$-Factors in Graphs with Low Independence Number}

\newcommand*\samethanks[1][\value{footnote}]{\footnotemark[#1]}
\author{ 
Charlotte Knierim%
\thanks{Department of Computer Science, ETH Zurich, Switzerland.
        Email: \{cknierim$\vert$sup\}@inf.ethz.ch} 
\and
Pascal Su%
\samethanks[1]%
\thanks{author was supported by grant no. 200021 169242 of the Swiss National Science Foundation}
}

\date{}

\begin{document}
	\maketitle

\begin{abstract}
A classical result by Hajnal and Szemer\'edi from 1970 determines the minimal degree conditions necessary to guarantee for a graph to contain a $K_r$-factor. Namely, any graph on $n$ vertices, with minimum degree $\delta(G) \ge \left(1-\frac{1}{r}\right) n $ and $r$ dividing $n$ has a $K_r$-factor. This result is tight but the extremal examples are unique in that they all have a large independent set which is the bottleneck. 
Nenadov and Pehova showed that by requiring a sub-linear independence number the minimum degree condition in the Hajnal-Szemer\'edi theorem can be improved.
We show that, with the same minimum degree and sub-linear independence number, we can find a clique-factor with double the clique size. More formally, we show for every $r\in \NN$ and constant $\mu>0$ there is a positive constant $\gamma$ such that every graph $G$ on $n$ vertices with $\delta(G)\ge \left(1-\frac{2}{r}+\mu\right)n$ and $\alpha(G)< \gamma n$ has a $K_r$-factor. We also give examples showing the minimum degree condition is asymptotically best possible.
\end{abstract}


\section{Introduction}

 Given two graphs $H$ and $G$, a collection of vertex-disjoint copies of $H$ in $G$ is called an $H$-tiling. A perfect $H$-tiling of $G$, an $H$-factor for short, is an $H$-tiling that covers all the vertices of $G$. Note that a perfect matching corresponds to a $K_2$-factor, thus the notion of $H$-factors is a natural generalization from edges to arbitrary graphs. In extremal graph theory there is a big interest in finding necessary or sufficient conditions for the existence of spanning substructures. Perfect matchings and Hamilton cycles are two commonly studied examples. In particular for perfect matchings, necessary and sufficient conditions are well-known by Hall's and Tutte's theorems.
 
Often, \emph{global} properties such as factors and Hamilton cycles have \emph{local} necessary conditions. Dirac \cite{dirac1952some} showed that if an $n$-vertex graph $G$ has minimum degree at least $n/2$, then it has a Hamilton cycle, in particular if $n$ is even then $G$ has a perfect matching. This was extended to triangle factors by Corr\'adi and Hajnal \cite{corradi1963maximal} in 1963 and later generalized to $K_r$-factors in a classical result by Hajnal and Szemer\'edi~\cite{hajnal1970proof}, who gave the sufficient minimum degree for $K_r$-factors.

\begin{theorem}[Hajnal and Szemer\'edi] \label{thm:hajnal}
    For every graph on $n$ vertices, given an integer $r\ge 2$, if $r$ divides $n$ and the minimum degree of $G$ is at least $\left(1-\frac{1}{r}\right)n$, then $G$ contains a $K_r$-factor.
\end{theorem}

A short proof was later found by Kierstead and Kostochka~\cite{kierstead2008short}. 
The divisibility condition in this theorem is necessary as the vertex set must be divisible by $|H|$ if we want to have an $H$-factor. The theorem is also tight in a sense that we can not lower the minimum degree condition and still hope to cover any $n$-vertex graph. 

Different results relating to the theorem of Hajnal and Szemer\'edi have been published. A degree sequence version of the result was published by Treglown~\cite{treglown2016degree} proving that, for a $({1}/{r})$-fraction of the vertices, the degrees can be smaller than prescribed by the Hajnal-Szemer\'edi theorem. Other results include the minimum degree condition in a 3-partite \cite{magyar2002tripartite}, 4-partite \cite{martin2008quadripartite} or multi-partite \cite{keevash2013multipartite} host graph. In each of these results, the known extremal examples all have one or more large independent sets. Naturally the question arises, what happens if we forbid these large independent sets?

 To cover any $n$-vertex graph with independence number\footnote{The independence number of a graph $G$, $\alpha(G)$, is the size of a largest independent set in $G$.} at least $n/r+1$ with cliques of arbitrary size we need at least $n/r+1$ cliques, as no clique can contain more than one vertex from the independent set. Taking an independent set of size exactly $n/r+1$ and adding edges from each of the remaining vertices to all other vertices gives a graph that does not have a $K_r$-factor and minimum degree $n-(n/r+1)=(1-1/r)n-1$.

So-called \emph{Ramsey-Tur\'an} type problems, first studied by Erd\H{o}s and S\'os \cite{erdos1970some} in 1970, ask for the minimum number of edges that force the existence of a given subgraph $H$ in a graph with bounded independence number. In particular $\textbf{RT}(H,o(n))$ denotes the smallest number of edges which guarantees that every graph $G$ on $n$ vertices with
$\alpha(G)=o(n)$ contains a copy of $H$. More on Ramsey-Tur\'an theory can be found e.g.\ in \cite{erdos1979turan,erdHos1983more,simonovits2001ramsey,SUDAKOV200399}.

 Continuing this line of research, Balogh, Molla and Sharifzadeh~\cite{balogh2016triangle} proved that the minimum degree requirement for a triangle factor in $G$ decreases if the independence number of $G$ is small showing that $\delta(G)\ge1/2+\varepsilon$ suffices in this case. Nenadov and Pehova~\cite{nenadov2018ramsey} extended their result to larger cliques and a generalization of the independence number.
 They show that instead of $\delta(G) \ge \left(1-\frac{1}{r}\right) n$ one only needs roughly $\delta(G) \ge \left(1-\frac{1}{r-1}\right) n$ for the existence of a $K_r$-factor if we restrict the independence number of $G$ to be sub-linear.

We further improve the minimum degree condition, doubling the clique size compared to the Hajnal-Szemer\'edi theorem. We see in the following that this is best possible. 
\begin{theorem}
	For every $r \ge 4$ and $\mu>0$ there are constants $\gamma$ and $n_0\in \mathbb{N}$ such that every graph $G$ on $n\ge n_0$ vertices where $r$ divides $n$, with $\delta(G)\ge \left(1-\frac{2}{r}+\mu\right)n$ and $\alpha(G)< \gamma n$ has a $K_r$-factor. 
	\label{thm:mainresult}
\end{theorem}

Note that the bound is not true for $r=2,3$. Balogh, Molla and Sharifzadeh~\cite{balogh2016triangle} observed that a minimum degree of $(1/2+\varepsilon)n$ is needed in the case $r=3$. This can be seen by considering graphs with a bipartition such that there are no triangles which span over both parts. In particular, for $n$ divisible by 4, the graph $K_{n/2+1} \cup K_{n/2-1}$, the union of two disjoint cliques, has independence number $2$ and minimum degree $n/2-2$ but does not contain a perfect matching nor a triangle factor because $n/2-1$ and $n/2+1$ are both odd and cannot both be divisible by three. Balogh, McDowell, Molla and Mycroft~\cite{balogh2018triangle} showed that the minimum degree condition can be lowered if an additional divisibility condition is added to avoid exactly this case.

The tightness of the Hajnal-Semer\'edi theorem came from large independent sets. So what is the bottleneck if we forbid these? By definition, $\alpha(G)=o(1)$ implies that every set of linear size has at least one edge inside, but if we have a large triangle-free set then we can take at most two vertices from this set for every clique. In particular, if we have an $n$-vertex graph with a triangle-free set of size $2n/r+1$ then we cannot hope to find a $K_r$-factor. The existence of triangle-free graphs with sub-linear independence number is related to the asymmetric Ramsey number $R(3,n)$. This is well studied, results can be found e.g.\ in \cite{erdos1961graph}, \cite{kim1995ramsey}.

The above construction shows that Theorem~\ref{thm:mainresult} is asymptotically tight. Look at the following example of an $n$-vertex graph. Take a triangle-free graph of size $2n/r +1$ and add $\left(1-\frac{2}{r}\right)n -1$ vertices each connected to all other vertices. The triangle-free subgraph of size $2n/r +1$ becomes a bottleneck since we can take at most to two vertices from it to complete to a $K_r$ and we cannot cover the graph with $n/r$ many $K_r$ (see Figure~\ref{fig:extremalexample}). So in this graph we have $\delta(G)> \left(1-\frac{2}{r} \right)n$ and $\alpha(G) = o(1)$ but there is no $K_r$-factor. This construction was first given in~\cite{balogh2016triangle}.

\captionsetup{justification=raggedright,singlelinecheck=false}
\begin{figure}
\centering
\includegraphics[scale=0.9]{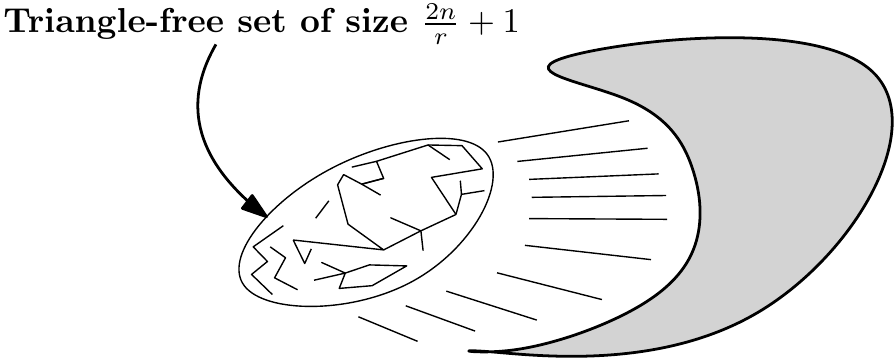} 
\caption{ An extremal example showing that we cannot improve the degree condition below $\left(1-\frac{2}{r}\right)n$.
}
\label{fig:extremalexample}
\end{figure}

\captionsetup{justification=centering}
Our proof combines well-known methods like the Regularity Lemma and embedding techniques with new ideas that use the low independence number.
Our paper is structured as follows. In Section~\ref{sec:preliminaries} we introduce some notation and definitions and give tools that will be useful in the later proofs. The remainder of the paper contains the proof of Theorem~\ref{thm:mainresult}.
The proof consists of two parts. 

First, in Section~\ref{sec:absorbers}, we use the absorbing method. 
This is a technique mainly pushed forward by Rödl, Ruci\'nski and Szemer\'edi~\cite{rodl2006dirac, rodl2006perfect,  rodl2008approximate}. The method implies that, under the appropriate circumstances, it is enough to find a $K_r$-tiling covering everything but a small fraction. The method sets aside a small set of vertices at the beginning which we can cover flexibly enough so that we can ``absorb'' any small fraction of the other vertices which may remain. For the more precise definition see Definition~\ref{def:absorber}.

Second, in Section~\ref{sec:almost}, we prove that the minimum degree and independence number conditions are enough to cover everything but a small $\xi$-fraction of all vertices with a $K_r$-tiling. This is also known as an almost cover of the vertices.
To show that there is an almost cover with $K_r$'s in the graph we find a fractional tiling in the reduced graph after applying the Regularity Lemma and convert this back. 

We adapt some well-known techniques to make use of the fact that the independence number of $G$ is low. Embedding independent sets into a cluster of the reduced graph of the Regularity Lemma is standard, but we sometimes want to embed edges instead of single vertices. In fact we use that the low independence number implies we can find paths of small length in any small linear sized subset of the vertices. We are required to differentiate between edges in the reduced graph which represent densities above $1/2 + \beta$ and those only above $\beta$. We believe this approach might also work for embedding other graphs into a host graph with a low independence number.

\section{Preliminaries}
\label{sec:preliminaries}
Since many of our constructions are specifically built for making use of the low independence number we first introduce some definitions. We also prove an embedding lemma that will be useful in multiple places to convert structures we found in the reduced graph back to the original graph. We begin with the notion of $\epsilon$-regular. Throughout the paper we use standard graph theory notation (see e.g.\ \cite{diestel2012graph}). For a graph $G$, we use $|G|$ to refer to the number of vertices in $G$. In general, variables represented by Latin letters will be variables in $\NN$ and variables represented by Greek letters will be small positive real numbers. $N_G(v)$ denotes the neighborhood of a vertex in $G$ which is the set $N_G(v) = \{ u \in V(G) | \{v, u\} \in E(G) \} $. We omit the index $G$ if the graph is clear from the context. On the contrary $\deg_G(v)$ denotes the number of outgoing edges from $v$ in the graph $G$, counting a double-edge twice in the case of multigraphs. We write $\deg_G(v, S)$ if we restrict to edges to a subset $S$ of $V(G)$ and again we omit the index if not needed. Further for two vertex sets $U$ and $W$, we denote by $\deg_G(U, W)$ the combined degrees over all vertices in $U$, $\sum_{u\in U}\deg_G(u, S)$

\begin{definition}[$\epsilon$-regular]
	Given a graph $G$ and disjoint subsets $V_1, V_2 \subseteq V (G)$, we say that the pair
$(V_1, V_2)$ is $\epsilon$-regular if for all $X\subseteq V_1, |X| \ge \epsilon |V_1|$ and $Y\subseteq V_2, |Y | \ge \epsilon|V_2|$ we have
$|d(X, Y ) - d(V_1, V_2)| \le \epsilon$
where $d(X, Y ) = \deg(X, Y ) / |X||Y|$
\end{definition}
The following fact is an easy consequence from the definition of regularity. It is sometimes known as the Slicing Lemma (cf.\cite{komlos96szemerediregularity}).

\begin{fact}
	Let $B=(V_1\cup V_2,E)$ be an $\varepsilon$-regular bipartite graph, let $\alpha>\varepsilon$ and let $V_1'\subset V_1$ and $V_2'\subset V_2$ be subsets with $|V_1'|\ge \alpha|V_1|$ and $|V_2'|\ge \alpha |V_2|$. Then for $\varepsilon'\ge \max\{\varepsilon/\alpha,2\varepsilon\}$ the graph $B'=B[V_1'\cup V_2']$ induced by $V_1'$ and $V_2'$ is $\varepsilon'$-regular with $|d_B(V_1,V_2)-d_{B'}(V_1',V_2')|<\varepsilon$.
	\label{fact:regsub}
\end{fact}
Our proof builds upon the famous Regularity Lemma by Szemer\'edi. Originally from~\cite{szemeredi1975regular} there have been many variants making it slightly stronger or adapted to a particular problem. The following is the degree variant of the Regularity Lemma.

\begin{lemma}[Regularity Lemma \cite{komlos96szemerediregularity}, 
Theorem 1.10]
	For every $\epsilon>0$ there is an $M=M(\epsilon)$ such that if $G$ is a graph on $n\ge M$ vertices and $\beta \in[0,1]$ is a real number, then there exists a partition $V(G)=V_0\cup\ldots\cup V_k$ and a spanning subgraph $G'\subseteq G$ with the following properties:
	\begin{enumerate}
		\item \label{lem:reg:k} $k\le M$,
		\item \label{lem:reg:V0} $|V_0|\le \epsilon n$,
		\item \label{lem:reg:Vi} $|V_i|=m$ for all $1\le i\le k$ with $m\le \epsilon n$,
		\item \label{lem:reg:deg}$\deg_{G'}(v)>\deg_G(v)-(\beta+\epsilon)n$ for all $v\in V(G)$,
		\item \label{lem:reg:indsets} $V_i$ is an independent set in $G'$ for all $i\in [k]$,
		\item \label{lem:reg:pairs} all pairs $(V_i,V_j)$ are $\epsilon$-regular with density $0$ or at least $\beta$. 
	\end{enumerate}
	\label{lem:reg}
\end{lemma}

What is new in our case is that we must differentiate between dense and very dense pairs of partitions. The following definition replaces the usual reduced graph of the Regularity Lemma. We call it the reduced multigraph throughout the paper.

\begin{definition}[reduced multigraph] \label{def:reducedmultigraph}
	For a graph $G$ and $\beta,\epsilon>0$ let $V(G)=V_0\cup\ldots\cup V_k$ be a partition and $G'\subseteq G$ and a subgraph fulfilling the properties of Lemma~\ref{lem:reg}. We denote by $R_{\beta, \epsilon}$ the \emph{reduced multigraph} of this partition, which is defined as follows. Let $V(R_{\beta, \epsilon})=\{1,\ldots,k\}$ and for two distinct vertices $i$ and $j$ we draw two edges between $i$ and $j$ if $d_{G'}(V_i,V_j)\ge 1/2+\beta$, one edge if $d_{G'}(V_i,V_j)\ge \beta$ and no edge otherwise. 
\end{definition}

In this reduced multigraph we sometimes refer to the vertices as clusters because of the correspondence to sets of vertices in the original graph.
We omit the subscripts $\beta$ and $\epsilon$ whenever it is clear from the context or the parameters are not used.

The following fact connects a minimum degree condition in $G$ to a minimum degree condition in reduced multigraph.
\begin{fact}\label{fact:min_deg_r}
Let $G$ be a graph with $\delta(G)\ge \left(1-\frac{2}{r}+\mu\right)n$ and $ V_1\cup\ldots\cup V_k$ be the partition given by the Regularity Lemma with the corresponding reduced multigraph $R_{ \beta, \epsilon}$ for $\epsilon$ and $\beta$ smaller than $\mu/10$. Then for every $i\in V(R_{\beta, \epsilon})$ we have
\[\deg_{R_{\beta, \epsilon}}(i)\ge  2\left(1-\frac{2}{r}+\mu/2\right)k.\]
\end{fact}
\begin{proof}
    For every $i\in V(R_{\beta, \epsilon})$ we have $\deg_{G'}(V_i, \bigcup_{j\ne i} V_j)$ is at least \[|V_i|\left(\left(1-\frac{2}{r}+\mu -(\epsilon+\beta) \right) n -|V_0| \right) \ge \left(1-\frac{2}{r}+\mu -2\epsilon -\beta) \right) nm .\]
    Every edge in $R_{\beta, \epsilon}$ represents less than $(1/2 + \beta) m^2$ edges in $G' \setminus {V_0}$. So $R_{\beta, \epsilon}$ must have minimum degree at least
\[\deg_{R_{\beta, \epsilon}}(i)\ge \frac{ \left(1-\frac{2}{r}+\mu -2\epsilon - \beta\right) nm}{(1/2+ \beta )m^2}\ge 2\left(1-\frac{2}{r}+\mu/2\right)k,\]
where in the last step we use the upper bounds on $\beta$ and $\epsilon$, $m\le n/k$ and $(1/2+\beta)^{-1}\ge 2(1-2\beta)$.\end{proof}

Then to formalize the intuition of embedding two vertices into a cluster of the reduced graph we define a multi-embedding.

\begin{definition}[$H$-multi-embedding]\label{def:embedding}

	Let $R$ be a reduced multigraph. We say that a simple graph $H$ is embeddable into the multigraph $R$ if there is a mapping $f:V(H)\to V(R)$ such that the following holds:
	\begin{enumerate}
		\item \label{def:emb:1} For any $i\in V(R)$ the induced subgraph on the vertices $f^{-1}(i)$ in $H$ is either an isolated vertex, an edge or a path of length 2. 
		\item \label{def:emb:2}  If $\{u,v\}\in E(H)$, then $f(u)$ and $f(v)$ are connected by at least one edge in $R$ (as long as $f(u)$ and $f(v)$ differ).
		\item \label{def:emb:3}  If for $i,j\in V(R)$ we have that $f^{-1}(i)$ and $f^{-1}(j)$ have at least two vertices and are connected in $H$, then $i$ and $j$ are connected with two edges. 
		
		\item \label{def:emb:4}  The joint neighborhood of the vertices embedded into a single cluster has at most two vertices embedded in any other cluster. That is
		\[ | N_{H}(f^{-1}(i))   \cap f^{-1}(j)   |  \le 2 \qquad \forall  i, j \in V(R) \  (i \ne j) \]
		where $N_{H}(f^{-1}(i)) = \bigcup_{w \in f^{-1}(i)} N_{H}(w)$ is the combined neighborhood of all vertices embedded in cluster $i$.
	\end{enumerate} 
	
	We call $f$ a multi-embedding of $H$.
\end{definition}

\begin{figure}[b]

\centering
\includegraphics[scale=0.6]{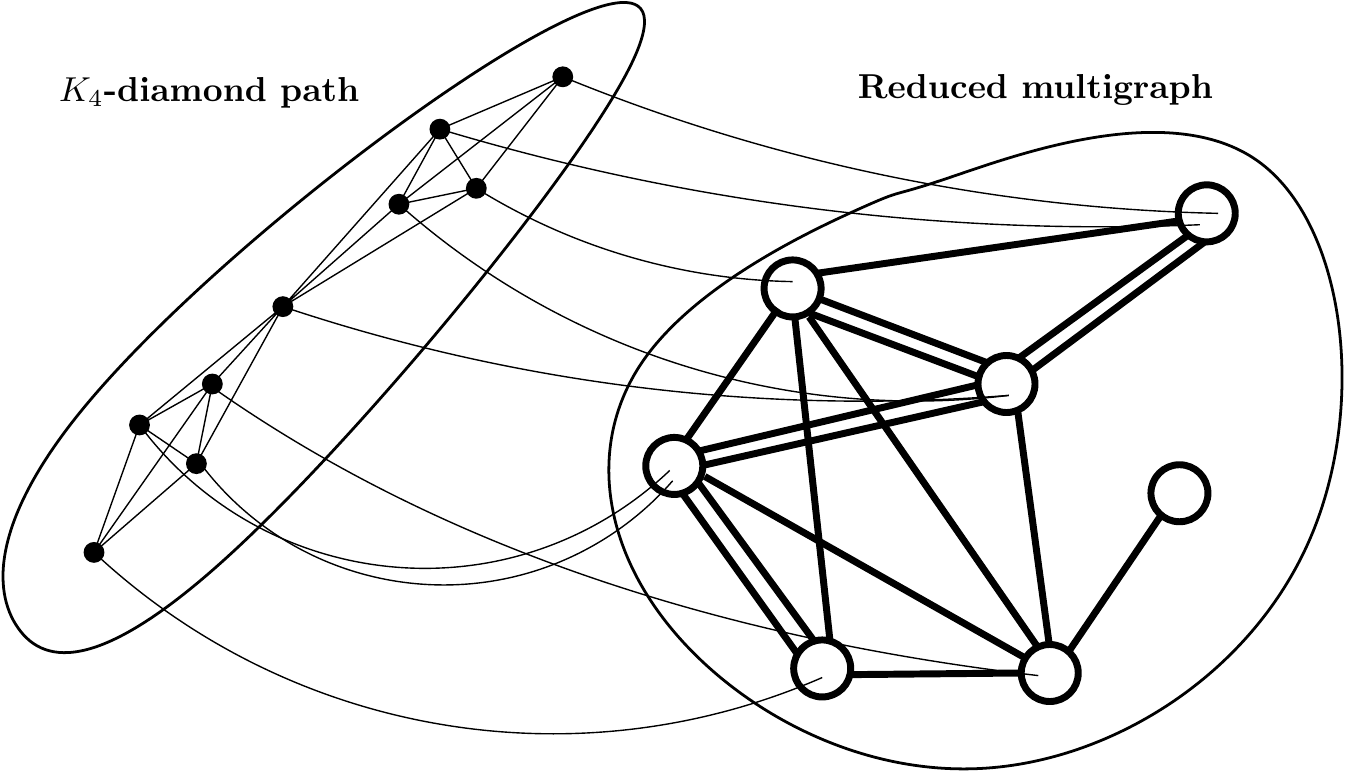}\\ 

\label{fig:embedding}
\caption{Multi-embedding of a $K_{4}$-diamond path}
\end{figure}

We prove that this embedding is useful in the intended way. Whenever we can find a multi-embedding of a graph $H$ in a reduced multigraph of $G$ then we can also find many copies of $H$ as a subgraph of $G$.
\begin{lemma}[Embedding Lemma]\label{lem:embedable_struct}
	For every graph $H$ with $|H| = h $ and $\beta >0$ there exist $\epsilon,  \gamma> 0$ and $n_0 \in \NN$ such that the following holds for every graph $G$ on $n > n_0$ vertices and with independence number $\alpha(G) \le \gamma n$ and the sets $ V_1\cup\ldots\cup V_k$ with $|V_i|=m$ given by the Regularity Lemma with the corresponding reduced multigraph $R_{ \beta, \epsilon}$.
	Let $f$ be a multi-embedding of $H$ into $R_{ \beta, \epsilon}$ with $f(V(H))=\mathcal{I}=\{i_1,\ldots,i_t\}$ for some $1\le t\le |H|$. Then let $V'_{i_1},\ldots,V'_{i_t}$ be subsets of $V_{i_1},\ldots,V_{i_t}$ respectively of size at least $({2 }/ {\beta})^{h} \epsilon m$. There exists a copy of $H$ as a subgraph of $G$ such that $v \in V'_{f(v)}$ for each vertex $v$ in $H \subset G$.
\end{lemma}
\begin{proof} 
	Ensure $\beta^{h}\ge (h+1)\varepsilon$ and $\epsilon m \ge 3 \gamma n$. The subgraph can be chosen greedily and we show this by induction on the size of $\mathcal{I}$. 
	The base case is clear, if we only have $| \mathcal{I}| = 1$ the multi-embedding can be at most a path of length two. But since $V_{i_1}'$ is of size at least $({2 }/ {\beta})^{h} \epsilon m \ge 3 \gamma n \ge 3 \alpha(G)$ there is always a path of length two in $V_{i_1}'$. 
	
	For the induction we first check if there is any vertex of $R$ which has a single vertex embedded. If so choose one of these, say $i_v$ and $f^{-1}(i_v) = v \in H$. By the Property~\eqref{def:emb:2} of a multi-embedding, there are edges between cluster-vertices in $R$ if their vertices in $H$ are in the neighborhood of $v$. Consider any of these edges $\{i_w, i_v\}$. This means that $V_{i_w}$ and $V_{i_v}$ must be $\varepsilon$-regular with density at least $\beta$. This means that at most $\epsilon m$ vertices of $V'_{i_v}$ have a neighborhood smaller than $(\beta - \epsilon) |V'_{i_w}|$ in $V'_{i_w}$. This holds for any neighbor of $v$ in $H$. As $\epsilon m h < m$, there is at least one vertex in $V'_{i_v}$ which has at least $(\beta/2)  |V'_{i_j}|$ neighbors in $V'_{i_j}$ for all $i_j \in f(N_H(v))$, choose one arbitrarily say $s_v$. Choose $V''_{i_j}$ to be the neighborhood of $s_v$ if $f^{-1}(i_j)$ contains a neighbor of $v$ or set it equal to $V'_{i_j}$ if not. Note that $|V''_{j}| \ge ({{2 }/ \beta})^{h-1}\epsilon m $ for all $j \in \mathcal{I} $ and that $f$ restricted to $H \setminus v$ is still a multi-embedding into $\mathcal{I} \setminus i_v$. So we can apply the induction hypothesis and find the subgraph $H \setminus v$ of $G$ such that the vertices are chosen from $V''_{i_j}$. Since all of the necessary $V''_{i_j}$ are in the neighborhood of $s_v$, we have that the graph from the induction together with $s_v$ form $H \subset G$ as desired.
	 
    The case where each cluster has at least two vertices embedded works analogously. Choose a vertex in $\mathcal{I} \subset R$ arbitrarily, say $i_v$, and let $f^{-1}(i_v)$ be the vertices $v_1, v_2$ and possibly $v_3$ of $H$. We call $f(N_H( f^{-1}(i_v))) \subset R$, excluding $i_v$, the set of corresponding neighbors of $i_v$. Because there are double-edges between $i_v$ and its corresponding neighbors, for any of the corresponding neighbors $i_w$ we have that $V_{i_v}$ and $V_{i_w}$ are regular with density at least $1/2 + \beta$.
    So all but at most $\epsilon m$ vertices of $V'_{i_v}$ have degree at least $(1/2 + \beta  - \epsilon) |V'_{i_w}|$ in $V'_{i_w}$. Since removing these bad vertices, which are at most $ \epsilon m h$ many, from $V'_{i_v}$ still leaves us with at least $3 \gamma n$ vertices and there must exist a 2-path (or edge). We choose one of these arbitrarily. By Property~\eqref{def:emb:4} of the multi-embedding, at most two of its vertices need to suffice a neighboring condition to any other cluster $V'_{i_w}$ and since the degree of each is at least $(1/2 + \beta -\epsilon) |V'_{i_w}|$ also the common neighborhood of these two vertices is larger than $(\beta /2) |V'_{i_w}|$. Take this neighborhood to be $V''_{i_w}$ for all corresponding neighbors of $i_v$ (and $V'_{j} = V''_{j}$ where there is no neighborhood condition to be fulfilled). We apply the induction hypothesis on the remaining graph $H \setminus \{v_1,  v_2, v_3 \}$ with its restricted multi-embedding and the sets $V''_{j} \ \forall j \in \mathcal{I} \setminus \{i_{v_1}, i_{v_2}, i_{v_3} \}$ to find $H \setminus \{v_1,  v_2, v_3 \}$ as a subgraph of $G$ which we can extend by the path we chose in $V'_{i_v}$ to get graph $H \subset G$. 
\end{proof}

\begin{corollary}	For every graph $H$ with $|H| = h $ and $\beta >0$ there exist $\epsilon,  \gamma> 0$ and $n_0 \in \NN$ such that the following holds for every graph $G$ on $n > n_0$ vertices and with independence number $\alpha(G) \le \gamma n$ and the sets $ V_1\cup\ldots\cup V_k$ with $|V_i|=m$ given by the Regularity Lemma with the corresponding reduced multigraph $R_{ \beta, \epsilon}$.
	Let $f$ be a multi-embedding of $H$ into $R_{ \beta, \epsilon}$ with $f(V(H))=\mathcal{I}=\{i_1,\ldots,i_t\}$ for some $1\le t\le |H|$. Then let $V'_{i_1},\ldots,V'_{i_t}$ be subsets of $V_{i_1},\ldots,V_{i_t}$ respectively of size at least $({2 }/ {\beta})^{h} \epsilon m$.  Additionally, let $u,v\in V(H)$ and $u_G,v_G\in V(G)$. Then there is an embedding of $H$ in $G$ such that $u$ is mapped to $u_G$ and $v$ is mapped to $v_G$ if the following holds.

    \begin{enumerate}[(i)]
        \item There is a multi-embedding $f$ of $H \setminus \{ u, v\}$ into $R_{\beta, \epsilon}$.
        \item For all edges of the form $\{u, x\}$ and $\{v, y\}$ in $H$ also $deg(u_G, V_{f(x)})\ge \beta |V_{f(x)}|$ and $deg(v_G, V_{f(y)}) \ge \beta |V_{f(y)}|$ in $G$ respectively.
        \item $u$ and $v$ have distance at least 3 in $H$.
    \end{enumerate}
    \label{cor:embedable_struct} 
\end{corollary}

\begin{proof}
    The embedding works the same as Lemma~\ref{lem:embedable_struct}. First fix $u$ and $v$ as $u_G$ and $v_G$, and then for each neighbor $x$ of $u$ choose  $V'_{f(x)} = N(u_G) \cap V_{f(x)}$, same for neighbors of $v$. For all other vertices in $H$ simply choose $V'_{f(i)} = V_{f(i)} $. So all $|V'_i| \ge \beta |V_i|$ and by Lemma \ref{lem:embedable_struct} we can find an embedding of $H\setminus \{u, v\}$ and the embedding of neighbors $u$ and $v$ will also be neighbors of $u_G$ and $v_G$ respectively. The distance $3$ is used to ensure no vertex in $H$ is neighbor to both $u$ and $v$.
\end{proof}

\section{Absorbers} \label{sec:proof_strategy} 
\label{sec:absorbers}

Absorbers are a well known tool and they allow us to prove statements about spanning subgraph structures. Often when working with the Regularity Lemma, we only find subgraph structures which cover almost all of the vertices, so all but a small linear fraction. Absorbers allow us to go the last step, they are structures we set aside in advance and which can ``absorb'' this small fraction of leftover vertices.

\begin{definition}  \label{def:absorber}
  Let $H$ be a graph with $h$ vertices and let $G$ be a graph with $n$ vertices.
  \begin{itemize}
    \item We say that a subset $A \subseteq V(G)$ is \emph{$\xi$-absorbing} for some $\xi > 0$ if for every subset $R \subseteq V(G) \setminus A$ such that $h$ divides $|A| + |R|$ and $|R| \le \xi n$ the induced subgraph $G[A \cup R]$ contains an $H$-factor.

    \item Given a subset $S \subseteq V(G)$ of size $h$ and an integer $t \in \NN$, we say that a subset $A_S \subseteq V(G) \setminus S$ is \emph{$(S, t)$-absorbing} if $|A_S| = h t$ and both $G[A_S]$ and $G[A_S \cup S]$ contain an $H$-factor.
  \end{itemize}
\end{definition}

We use the following lemma which gives a sufficient condition for the existence of $\xi$-absorbers based on abundance of disjoint $(S, t)$-absorbers.
The proof of Lemma \ref{lemma:absorbing} is based on ideas of Montgomery \cite{montgomery2014embedding} and relies on the existence of `robust' sparse bipartite graphs. 

\begin{lemma}[Nenadov, Pehova \cite{nenadov2018ramsey}] \label{lemma:absorbing} 
    Let $H$ be a graph with $h$ vertices and let $\phi > 0$ and $t \in \NN$. Then there exists $\xi$ and $n_o\in \NN$ such that the following is true. Suppose that $G$ is a graph with $n \ge n_0$ vertices such that for every $S \in \binom{V(G)}{h}$ there is a family of at least $\phi n$ vertex-disjoint $(S, t)$-absorbers. Then $G$ contains an $\xi$-absorbing set of size at most $\phi n$.
\end{lemma}

We define the following structure as it will be used as the main building block in the remainder of this section.

\begin{definition}[$K_{r}$-diamond path]
A $K_{r}$-diamond path between vertices $u$ and $v$ is the graph formed by a sequence of disjoint vertices $u = v_1, v_2, .. , v_\ell = v$ and disjoint cliques of size $r-1$ in the joint neighborhood of each pair of consecutive vertices. The length of the $K_{r}$-diamond path is $\ell$, the number of vertices in the sequence. 
\end{definition}

\begin{figure}[ht]
\centering
\includegraphics[scale=1]{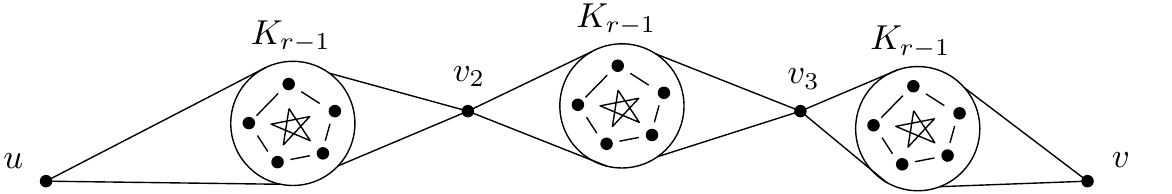} 
\label{fig:diamondpath}
\caption{$K_{r}$-diamond path}
\end{figure}

\subsection{Finding $K_{r}$-diamond paths}

To make use of this lemma we additionally need to find vertex-disjoint $(S, t)$-absorbers in our graph. Observe that if we can find a $K_r$ with disjoint $K_r$-diamond paths attached to each of its vertices, then this structure is ($S,t$)-absorbing for the set $S$ of $r$ free endpoints of the $K_r$-diamond paths. To find vertex-disjoint $(S, t)$-absorbers it is sufficient to find many disjoint $K_{r}$-diamond paths between any two vertices. 

\begin{lemma} \label{lem:absorber_diamondpaths}
For every $r \ge 4$ and $\mu>0$ there exist $\gamma >0$ and $n_0\in \NN$ such that in every graph $G$ on $n\ge n_0$ vertices with minimum degree $\delta(G)\ge \left(1-\frac{2}{r}+\mu\right)n $ and $\alpha(G)\le \gamma n$, after deleting $(\mu/2) n$ many vertices we can still find a $K_{r}$-diamond path of length at most $7$ between any two remaining vertices.
\end{lemma}
Note that, for connectivity issues, the lemma only holds for $r\ge4$. To prove this lemma we find a multi-embedding of a $K_{r}$-diamond path in a reduced multigraph and then extract from that a $K_{r}$-diamond path in the original graph. We introduce the notion of a $K_r$-neighborhood $\Upsilon_r(v)$. These are the neighbors of $v$ such that additionally we can find a multi-embedding of a $K_r$ into the reduced multigraph covering both the vertex and $v$.

\begin{definition}
Let $R$ be a reduced multigraph. Then for any vertex $v$ the \emph{$K_r$-neighborhood} $\Upsilon_r(v)$ is defined as follows.
 \begin{align*}
   \Upsilon_r(v) =  \{  w \in V(R) | \exists \text{ a multi-embedding } \psi : V(K_r) \rightarrow V(R)& \text{ s.t. }  \\
   \psi^{-1}(v) \ne\emptyset & \text{ and } \psi^{-1}(w) \ne \emptyset \} 
 \end{align*}
Further we define $\Upsilon^2_r(v) =  \bigcup_{u \in \Upsilon_{r}(v)} \Upsilon_{r}(u)$ to be the second-$K_r$-neighborhood.
\end{definition}
Note that by definition any vertex $v$ is in its own $K_r$-neighborhood assuming there is at least one $K_r$ multi-embedding containing $v$. Then also $\Upsilon_r(v) \subseteq \Upsilon^2_r(v)$.

In order to find $K_{r}$-diamond paths we first show that the $K_{r+1}$-neighborhood for every vertex in the reduced graph is large.

\begin{proposition} \label{prop:upsilonexpand} 
For $r\ge 4$, let $R$ be a reduced multigraph on $k$ vertices with $\delta(R) > (1-2/r)2 k$ then we have 
\[  \left| \Upsilon^2_{r+1}(v)  \right| \ge \frac{k}{2} \quad \quad \forall  v\in V(R). \]

\end{proposition}

Before we prove this proposition, we prove a series of lemmas about the size of $K_r$-neighborhoods. Note that this is easier for large $r$ thus we have to consider some special cases for small values of $r$. We start with some general lemmas that hold for all $r$. 

In the following, a clique of double-edges denotes a clique where all edges are double-edges and the double-edge-neighborhood of a vertex $v$ is the set of neighbors connected to $v$ with a double-edge. 
\begin{lemma} \label{lem:double-edgeconnection}
For $r\ge 4$, let $R$ be a reduced multigraph on $k$ vertices with $\delta(R) > (1-2/r)2 k$. For any vertex $v$ the vertices connected to $v$ by double-edges are contained in the $K_{r+1}$-neighborhood $\Upsilon_{r+1}(v)$.
\end{lemma}

\begin{proof}
In order to embed $K_{r+1}$ we need a clique of double-edges of size $\ell$ and a clique of size $r+1-2\ell$ in the neighborhood of this clique. Note that the double-edge itself is already a clique of size $2$. In the following, we show that for every $2\le \ell\le (r+1)/2$ we can find such an embedding given that $\ell$ is the size of a maximal clique of double-edges. 

Fix any double-edge of $v$ and take the largest clique of double-edges containing the double-edge. Let $\ell$ be the size of the clique, and let $S$ be the set of all vertices which lie in the joint neighborhood of all vertices of the clique. As we assumed the clique of double-edges to be maximal we know that every vertex in $S$ has at most $2\ell-1$ edges into the clique. Every vertex that is not in $S$ has to have at least one non-neighbor in the clique and can thus not have more than $2(\ell-1)$ edges into the clique. Moreover, by our minimum degree condition in $R$ we know that every vertex in the clique has at least $(1-r/2)2k $ edges. Combining this, we get
\[ (2\ell-1)|S|+2(\ell-1)(k-|S|) > \ell\left(1-\frac{2}{r}\right)2k, \]
from which we conclude that
\[|S| > 2k-\frac{4k\ell}{r} .\numberthis\label{eq:lowerS} \]
For any vertex $v \in R$ it holds that the neighborhood 
	\[|N(v)| \ge \deg(v) /2 > \left(1 - \frac{2}{r}\right) k. \]
In particular the number of vertices not in the neighborhood of a vertex is less than $\frac{2k}{r} $. So by greedily picking vertices one by one we can choose at least 
	\[\left\lceil  \frac{\ |S|\ }{\frac{2 k}{r}   } \right\rceil \ge r-2\ell+1 \]
many vertices. This gives us a clique of double-edges of size $\ell$ and in the joint neighborhood a clique of size $r-2\ell+1$ into which we can find a multi-embedding of $K_{r+1}$.
\end{proof}

\begin{lemma} \label{lem:highsingle}
For $r\ge 3$, let $R$ be a reduced multigraph on $k$ vertices with $\delta(R) > (1-2/r)2 k$. For any vertex $v$ in the $R$, if the neighborhood of $v$ is of size at least $ (1-1/r)k$, then $N(v) \subseteq \Upsilon_{r+1}(v)$.

\end{lemma}
\begin{proof}
We apply induction on $r$ by looking at the neighborhood of a vertex finding that the appropriate minimum degree conditions hold. 
The lemma is true for $r= 3$ since then any neighbor $u$ of $v$ has at least one vertex $w$ in the joint neighborhood with $v$ and we can create a multi-embedding $\psi$ which maps one vertex of a $K_4$ to $v$ and $u$ and maps the two remaining vertices to $w$. This is a valid multi-embedding of a $K_{4}$ and proves $N(v) \subseteq \Upsilon_{4}(v)$. This builds our induction base.

For $r>3$ consider for any vertex $u\in N(v)$ the joint neighborhood with $v$.
 \[ \deg(u, {N(v)}) > (1-2/r) 2k - 2(k - |N(v)|) \ge  (1-2/(r-1)) 2|N(v)|  ,   \]
 where in the last step we use that $|N(v)| \ge \frac{r-1}{r} k$. To prove that $u \in \Upsilon_{r+1}(v)$ it suffices to show that there is a $K_r$ multi-embedding containing $u$ in $R[N(v)]$, the subgraph induced by $N(v)$. Now $\delta(R[N(v)]) > (1-2/(r-1)) 2|N(v)|$ so for any vertex $u \in R[N(v)]$, by counting the edges, there must be either a double-edge containing $u$, in which case Lemma~\ref{lem:double-edgeconnection} gives at least one $K_r$ multi-embedding, or $u$ has a large neighborhood, $(1-2/(r-1))2|N(v)| \ge (1-1/(r-1))|N(v)|$, in which case we apply the induction on the subgraph $R[N(v)]$ so, in fact, in the subgraph $R[N(v)]$ any neighbor of $u$ is in $\Upsilon_{r}(u)$ and also $u$ is contained in a $K_r$.
\end{proof}

\begin{lemma} \label{lem:requal4}For $r = 4$, let $R$ be a reduced multigraph on $k$ vertices with $\delta(R) > (1-2/r)2 k$, then $N(v) \subseteq \Upsilon_{r+1}(v)$.
\end{lemma}

\begin{proof}
For any neighbor of $v$ we want to find a multi-embedding of $K_5$ mapping to $v$ and that neighbor.

By Lemma~\ref{lem:double-edgeconnection}, every double-edge-neighbor of $v$ is in $\Upsilon_{r+1}(v)$. For all other vertices $w \in N(v)$ we claim that either there is an edge between $w$ and a vertex $x$ in the double-edge-neighborhood of $v$, in which case we can map two vertices to $x$, two vertices to $v$ and one to $w$ to get a multi-embedding of $K_5$ or, in the other case, there is a double-edge between $w$ and another vertex $x$ in $N(v)$ and then we can map two vertices to $x$ and $w$ and one to $v$ to get a multi-embedding of $K_5$.

Let $D$ be the double-edge-neighborhood of $v$ and $S = N(v) \backslash D$. Then for any vertex $w \in N(v)$, if $w$ has no edge to any vertex in $D$ and at most one edge to any vertex in $S$, then 
\[ \deg(w) \le |S| + 2(k - |S| - |D|) \le 2k - (2|D| + |S|) \osref{$(2|D|+|S|)=\deg(v)$}< k   \]
which is a contradiction to the assumption that every vertex in the reduced graph $R$ has degree greater than $(1-2/r)2k = k$ for $r=4$.
\end{proof}

\begin{lemma} \label{lem:requal5}
For $r = 5$, let $R$ be a reduced multigraph on $k$ vertices with $\delta(R) > (1-2/r)2 k$ then $N(v)  \subseteq \Upsilon^2_{r+1}(v)$.
\end{lemma}
\begin{proof} Let $D$ be the double-edge-neighborhood of $v$. If $|D| \le 2k/5$ we have $N(v)  \ge 4k/5$ and by Lemma \ref{lem:highsingle} again we have that $N(v) \subseteq \Upsilon_{r+1}(v)$, so we assume $|D| \ge 2k/5$.

Look at any fixed $u \in N(v)\setminus\Upsilon^2_{r+1}(v)$. If $u$ has a double-edge to $D$, then by Lemma~\ref{lem:double-edgeconnection} it has distance two with regards to the $K_r$-neighborhood $\Upsilon_{r+1}$ and we are done. So we can assume it has only single edges or no edges to vertices in $D$. In particular, the double-edge-neighborhood of $u$ does not contain $D$ so its size is at most $k-|D|$. 

So since $|D| \ge 2k/5$, we have that 
\[N(u) \ge \frac{6k}{5} - \left(k-|D|\right) \ge \frac{k}{5} + |D| \ge \frac{7k}{5} - N(v), \]
where the last step follows from $N(v) > 6k/5 - |D|$.
Since the minimum neighborhood of any other vertex is $3k/5$, the common intersection of $u$ with any other vertex $x$ must be more than $7k/5 - N(v)+ 3k/5 -k = k - N(v)$, so we can choose a vertex $y$ in the joint neighborhood of $v$, $u$ and $x$.

Now choose $x$ in $D$. The multi-embedding of $K_6$ follows by embedding two vertices each in $v$ and $x$ one each in $u$ and $y$. So then $u \in  \Upsilon_{r+1}(v)$.
In any case  $u \in \Upsilon^2_{r+1}(v)$ and the lemma follows.

\end{proof}
Combining the previous lemmas, we are now ready to prove Proposition \ref{prop:upsilonexpand}.
\begin{proof}[Proof of  Proposition \ref{prop:upsilonexpand}]
For $r \ge 8$ the double-edge-neighborhood of every vertex is greater than $k/2$ so by Lemma \ref{lem:double-edgeconnection} this follows immediately. For $r \ge 6$ by looking at the degree, for each vertex either the double-edge-neighborhood is greater than $k/2$ or the total neighborhood is greater than $(1-1/r)k$, so by Lemma \ref{lem:double-edgeconnection} or Lemma \ref{lem:highsingle} the proposition follows. For $r = 4,5$  we have Lemmas \ref{lem:requal4} and \ref{lem:requal5} respectively,  where in both cases it is easy to see that $|N(v)|\ge k/2$.
\end{proof}

The next lemma is about connecting one fixed vertex $v$ in $G$ to $K_{r+1}$-embedable structures as follows. Given $v$, we want to find a multi-embedding of $K_{r-1}$ into the neighborhood of $v$ i.e.\ clusters that $v$ has many edges to. We then want to extend this $K_{r-1}$ to a $K_{r}$ by finding a vertex in the joint neighborhood of the clique (not necessarily in $N(v)$). This is a preparation step to apply Corollary~\ref{cor:embedable_struct}. 

\begin{lemma} \label{lem:start} 
     Fix a vertex $v$ in $G$ and a reduced multigraph $R_{\beta, \epsilon}$ of $G$. Let $Q_v$ be the set of vertices $i \in V(R_{\beta, \epsilon})$ such that for their corresponding clusters $V_i \subseteq V(G)$ it holds that $\deg(v,V_i) \ge \beta |V_i|$. Then there exists a multi-embedding of a $K_{r}$ into $R_{\beta, \epsilon}$ embedding at most one vertex into $V(R_{\beta, \epsilon}) \setminus Q_v$.
\end{lemma}

\begin{proof}
    Note that the number of edges from $v$ to $V_0$ or any cluster not in $Q_v$ is at most $\epsilon n$ and $\beta k m \le \beta n$ respectively. The degree of $v$ is at least $(1-2/r+\mu)n$ in $G$ and choosing $\beta, \epsilon< \mu/10$ the number of edges from $v$ to clusters of $Q_v$ is at least $(1-2/r + 2\mu/3) n$. In particular since every cluster has size at most $n/k$ this means   \[|Q_v| \ge \left(1-\frac{2}{r}+\mu/2 \right)k.   \numberthis\label{eq:lowerQv}\]  
    The proof follows similar arguments as the proof of Lemma~\ref{lem:double-edgeconnection}.
    Let the largest clique with double-edges in $Q_v$ be $C$ of size $\ell$. Let $S \subseteq Q_v$ be the joint neighborhood of the vertices from this clique inside $Q_v$ and $T \subseteq V(R) \setminus Q_v$ all vertices which are in the joint neighborhood of the clique but not in $Q_v$. We want to find a $K_{r-2\ell}$ in $S\cup T$ with at most one vertex in $T$. 
    
    Because $C$ is maximal every vertex in $S$ has at most $2\ell -1$ edges to $C$ and every other vertex in $Q_v$ has at most $2\ell -2$ edges to $C$. But also every vertex in $C$ has degree greater than $ (1-\frac{2}{r})2k $. So we get two bounds for $deg(C, {Q_v})$. the sum of degrees between $C$ and $Q_v$.
	\[deg(C, {Q_v})  > 	\ell  \left( (1-2/r )2k - 2(k-|Q_v|) \right)  >  \left(1- 2 / (r-2) \right)2  \ell |Q_v|, \]
	where in the last step we use from \eqref{eq:lowerQv} that $k < \frac{r}{r-2}|Q_v|$.
    
    \begin{align*}  
	\deg(C, {Q_v}) &  \le 	\deg (C, S) + \deg(C, {Q_v\backslash S}) \\
	& \le (2\ell-1)|S|+(2\ell-2)(|Q_v|-|S|) \\
	& < |S|+(2\ell-2)|Q_v|.
	\end{align*} 
    Together we get a bound on $|S|$. Namely 
    \[ |S| > \left(1- 2\ell / (r-2) \right)2|Q_v|  > (r-2\ell-2) \frac{2k}{r} \numberthis\label{eq:lowerSv}.\]
    
	Next, we bound the size of $S\cup T$ with a similar argument. Counting the edges $\deg(C, {V(R)})$. Again, vertices in $ V(R)$ but not in  $S\cup T$ can have at most $2\ell -2$ edges to $C$.
	
    \begin{align*}  
	(1-2/r)2k\ell & \le \deg(C,{V(R)}) \\ &  = 	\deg(C, S)	+ \deg(C, T) + \deg(C, {V(R)\setminus (S\cup T)}) \\
	& \le (2\ell-1)|S|+ 2\ell |T|+(2\ell-2)(k-|S| -|T|) \\
	& < |S|+ 2|T| +(2\ell-2)k.
	\end{align*} 
    We get a bound on $|S \cup T|$. Namely $|S| + 2|T| > (1-2\ell/r)2k  = (r-2\ell) \frac{2k}{r}  $ and in particular since $|T| \le k-|Q_v| < 2k/r$ because of (\ref{eq:lowerQv}) this means 
	    \[|S| + |T| > (1-2\ell/r)2k -2k/r = (r-2\ell-1) \frac{2k}{r}.  \numberthis\label{eq:lowerSTv} \]
	Furthermore, observe that for every vertex $w \in V(R)$ we have $N(w) \ge k - 2k/r$, thus every vertex has at most $2k/r$ non-neighbors. This directly implies we can sequentially choose 
		\[ \left\lceil \frac{\ |S|\ }{\frac{2k}{r}} \right\rceil \osref{(\ref{eq:lowerSv})}\ge r-2\ell-1 \] 
	many vertices from $S$ to form a clique and still have at least one vertex from $S\cup T$ because of (\ref{eq:lowerSTv}) to form the $K_{r-2\ell}$. This together with the $K_\ell$ of double-edges gives allows for a multi-embedding of $K_r$ and concludes the proof.
\end{proof}

We now prove Lemma \ref{lem:absorber_diamondpaths}.

\begin{proof}[Proof of Lemma~\ref{lem:absorber_diamondpaths}]
    Choose two arbitrary vertices $s,t \in V(G)$ for which we want to find a $K_r$-diamond path. For $s$ and $t$, apply Lemma~\ref{lem:start} to find two multi-embeddings of $K_{r}$'s such that at most one of the vertices in $R$ has $\deg(s, V_i) < \beta |V_i|$ and $\deg(t, V_j) < \beta |V_j|$ respectively. Call these vertices $s_1$ and $t_1$ respectively. With Proposition~\ref{prop:upsilonexpand} we find a multi-embedding of at most four $K_{r+1}$'s connecting $s_1$ and $t_1$ since the second $\Upsilon_{r+1}$ neighborhoods overlap.

    This almost gives a multi-embedding of a $K_r$-diamond path connecting $s$ and $t$. It remains to deal with the multi usage of a cluster in the reduced graph. For the mapping to be a multi-embedding as in Definition \ref{def:embedding} we need that each vertex/cluster in the reduced graph has only a single vertex, edge or 2-path mapped to it. For this we partition each cluster arbitrarily into enough parts such that we can assign each vertex, edge or 2-path to a unique part. Note that as we have at most six $K_r$'s we only need to split the clusters into constantly many parts. 
    
    If we arbitrarily split each cluster of the reduced graph into $ 6r$ equal parts, then the new partition still satisfies the conditions of the Regularity Lemma because $\varepsilon$-regularity is inherent by Fact~\ref{fact:regsub} just with slightly different $\varepsilon'$ and $\beta'$. So we can have a reduced multigraph $R'$ of this new partition which is just a blowup of $R$. In particular, we can embed each isolated vertex, edge or 2-path into a separate cluster. In $R'$ the consecutive $K_r$ multi-embedding is in fact a multi-embedding of a $K_r$-diamond path of length at most seven excluding the endpoints $s$ and $t$.
 It follows by Corollary~\ref{cor:embedable_struct} that we get a $K_r$-diamond path in $G$.
\end{proof}

With Lemma \ref{lem:absorber_diamondpaths} we can find $K_r$-diamond paths from any tuple of $r$ vertices matching them to a different vertex of a $K_r$ somewhere else in the graph. This is now a $(S,t)$-absorber from Definition~\ref{def:absorber} and together with Lemma~\ref{lemma:absorbing} this is enough to find an absorber of the first kind as in Definition~\ref{def:absorber}. 


\section{Almost Spanning Structure} \label{sec:almost}
For the second part of the proof we want to show that we can cover most of the vertices with a $K_r$-tiling. Combining this with the absorber gives a $K_r$-factor.

\begin{lemma}
	For every $r\in \NN$ and $\xi, \mu > 0$, there exist $\gamma > 0$ and $n_0 \in \NN$ such that every graph $G$ on $n>n_0$ vertices with $\delta(G)\ge \left(1-\frac{2}{r}+\mu\right)n$ and $\alpha(G) \le \gamma n$, we can find a $K_r$-tiling which covers at least $(1- \xi)n$ vertices in $G$.
	\label{lem:almost}
\end{lemma}

We make use of a known result for small subgraphs in the same setting. To find a $K_{r}$ in a graph with small independence number we only need a certain average degree. The following lemma states this 

\begin{lemma}[Erd\H{o}s, S\'os \cite{erdos1970some}]
	For every $r\in \NN$ and $\mu>0$ there exist $\gamma >0 $ and $n_0 \in \NN$ such that for every graph $G$ on $n>n_0$ vertices with \textbf{average degree} $d(G)\ge \left(1-\frac{2}{r-1}+\mu\right) n$ and $\alpha(G)\le \gamma n$, then $K_r\subseteq G$.
	\label{lem:RT}
\end{lemma}

First we would like to show, that there exists at least a fractional almost cover of the vertices. A fractional cover is defined as follows:

\begin{definition}
    A fractional $K_r$-tiling $\mathcal{T}$ of a graph $G$ is a weight function from the set $\mathcal{S}$ of all $K_r \subseteq G$ to the interval $[0, 1]$ such that for vertices of $G$ it holds that
    \[\hspace{6em}  w_{\mathcal{T}}( v) =  \sum_{\substack{\mathcal{K}_i \in \mathcal{S}, \\ v \in\mathcal{K}_i} } w_{\mathcal{T}}( \mathcal{K}_i ) \le 1 \hspace{6em} \forall v \in G. \]
    We call $\sum_{v \in G } w_{\mathcal{T}}(v)$ the total weight of a tiling and it is a perfect fractional tiling if equality holds for every vertex.
\end{definition}

Fractional $K_r$-tilings are somehow easier to find and we will prove the following lemma later in this section.
\begin{restatable}{lemma}{fracmat}
	For every $r\in \NN$ and $\eta$, $\mu >0$ there exist $\gamma >0$ and $n_0\in \NN$ such that every graph $G$ on $n\ge n_0$ vertices with $\delta(G)\ge \left(1-\frac{2}{r}+\mu\right)n$ and $\alpha(G)< \gamma n$ has a fractional $K_r$-tiling $\mathcal{T}$ such that 
	\[|\{v\in G\colon w_{\mathcal{T}}(v)<1-\eta \}|\le \eta n.\]
	\label{lem:fracmat}
\end{restatable}

Observe that the weight of this tiling is at least $(1-2\eta) n$.
We would like to transform the fractional into an actual tiling. We construct a fractional tiling in the reduced multigraph first, then transfer it to the original graph greedily. We will slightly abuse notation for the fractional tiling to extend the definition to the reduced multigraph. By a fractional tiling with $K_r$-embeddable structures we mean we assign the weights to all possible multi-embeddings of $K_r$ onto the reduced multigraph and require that the for every vertex all multi-embeddings mapping to that vertex have a total weight of at most one, counting multiplicity.

\begin{lemma}
	For every $r \in \NN$ and $\eta, \beta>0$ there exist $\epsilon, \gamma >0$ and $n_0\in \NN$ such that for every graph $G$ on $n\ge n_0$ vertices, if a reduced multigraph $R_{\beta, \epsilon}$ of $G$ has a fractional tiling with $K_r$-embeddable structures of total weight at least $(1-\eta) k$, then $G$ has an $K_r$-factor that covers all but $(1-2\eta)n$ vertices.
	\label{lem:frac_multi_to_almost}
\end{lemma}

\begin{proof}
    Set $\epsilon$ and $\gamma$ small enough for Lemma~\ref{lem:embedable_struct} and such that $(2/\beta)^r\epsilon \le \eta/2$. 
	The first step is to rescale the tiling. Let $\mathcal{T}$ be the fractional tiling of $R$ as given by the statement. Construct $\mathcal{T}'$ by scaling every $K_r$-embeddable structure with a factor of $(1-(2/\beta)^r\varepsilon)$ i.e.\ for any $K_r$-multi-embedding $\mathcal{K}$ we have $w_{\mathcal{T}'}(\mathcal{K})=(1-(2/\beta)^r\varepsilon)w_{\mathcal{T}}(\mathcal{K})$. We construct the $K_r$-tiling in $G$ by greedily taking $w_{\mathcal{T}'}(\mathcal{K})|V_i|$ many $K_r$ given by Lemma~\ref{lem:embedable_struct} and remove them from $G$. Note that, because of the rescaling, the sum of the weights of all $K_r$-embeddable structures touching one vertex is at most $(1-(2/\beta)^r\varepsilon)$. Thus, in every step of the greedy removal we have at least $(2/\beta)^r\varepsilon |V_i|$ vertices left which ensures that we can always apply Lemma~\ref{lem:embedable_struct}. Even after rescaling, $\mathcal{T}'$ has total weight at least $(1-(2/\beta)^r\varepsilon - \eta) k$ and $|V_0|$ has at most $\epsilon n$ many vertices. 
	So the greedy $K_r$-tiling of $G$ covers at least a $(1-((2/\beta)^r\epsilon + \eta + \epsilon)) \ge (1-2\eta)$ fraction of the vertices which concludes the proof.
\end{proof}

For our proof, we need triangle free graphs with low independence number that we can connect with relatively high density without creating a copy of $K_4$. A construction by Bollob\'as and Erd\H{o}s shows that these graphs exist. We state their results in a slightly different way, but it directly follows from their construction.
\begin{lemma}[\cite{bollobas1976ramsey}]
\label{lem:K4free}
For $\zeta,\gamma>0$ there is a $n_0\in \NN$ such that for $n\ge n_0$ there is a graph $G$ on $2n$ vertices with a split into $V_1, V_2$ has the following properties.
\begin{enumerate}
  \item $|V_1|=|V_2|=n$,
  \item $G[V_1]$ is isomorphic to $G[V_2]$ and they are triangle free,
  \item $G$ is $K_4$ free,
  \item $G[V_1,V_2]$ has density at least $1/2-\zeta$,
  \item $\alpha(G)\le \gamma n$
  
\end{enumerate}
\end{lemma}
\ \\
The next lemma connects almost tilings and fractional tilings. In order to find an almost tiling in a graph $G$ we apply the Regularity Lemma and need a fractional tiling in the reduced graph. We make use of a second auxiliary graph $\Gamma$, which is similar to a blow-up of the reduced graph. 
\begin{lemma}
	For every $r \in \NN$ and $\mu, \eta >0$ there exist $\beta, \epsilon, \gamma >0$ and $n_0\in \NN$ such that for every graph $G$ on $n\ge n_0$ vertices with minimum degree $\delta(G)\ge \left(1-\frac{2}{r}+\mu\right)n $ and $\alpha(G)\le \gamma n$ there is a graph $\Gamma $ with $\delta(\Gamma )\ge \left(1-\frac{2}{r}+\frac{\mu}{4} \right)|\Gamma | $ and $\alpha(\Gamma )\le \gamma |\Gamma |$ such that the following holds. 
	
	If $\Gamma $ has a fractional $K_r$-tiling with weight at least $(1-\eta)|\Gamma |$, then $G$ has a $K_r$-tiling covering at least $(1-2 \eta)n$ vertices.
	\label{lem:tilingtransferlemma}
\end{lemma}

\begin{proof}
    Choose $\beta$, $\varepsilon$ and $\gamma$ small enough such that Lemma~\ref{lem:embedable_struct} and Lemma~\ref{lem:frac_multi_to_almost} are satisfied and smaller than $\mu/10$. Apply the Regularity Lemma (Lemma \ref{lem:reg}) to $G$ with $\beta$ and $\varepsilon$. Let $V_0\cup V_1\cup\ldots\cup V_k$ be the regular partition resulting from the Regularity Lemma and let $R_{\beta, \epsilon}$ be the reduced multigraph of this partition. 
    Let $y_1$ be a constant that is larger than $n_0$ from Lemma \ref{lem:K4free} with $\gamma_{\ref{lem:K4free}} = \gamma$ and $\zeta_{\ref{lem:K4free}}=\mu/8$.
    Construct $\Gamma$ by taking $y_0 = k\cdot y_1$ vertices and split $V(\Gamma)$ into $W_1,\ldots,W_k$ each of size $y_1$ where we associate $W_i$ with $V_i$ from the regular partition. On every vertex set $W_i$ we put a triangle-free graph from Lemma~\ref{lem:K4free}. Then add a complete bipartite graph between two clusters $W_i$ and $W_j$ if $i$ and $j$ are connected by a double-edge in $R$. Add a $K_4$-free construction given by Lemma \ref{lem:K4free} between $W_i$ and $W_j$ if $i$ and $j$ are connected by a single edge and the empty graph otherwise. Note that as the graphs inside the clusters are all isomorphic, we are guaranteed that the $K_4$-free graph construction of Lemma \ref{lem:K4free} is possible between any two clusters.

    We consider the minimum degree of $\Gamma $. As $G$ has minimum degree $\delta(G) \ge (1-2/r+\mu)n$, using Fact~\ref{fact:min_deg_r} we get $\delta(R_{\beta,\epsilon})\ge 2\left( (1-2/r+\mu/2) \right) k$ which finally means in $\Gamma $ every edge from a cluster-vertex $i$ in $R_{\beta, \epsilon}$ contributes to at least $(1/2 - \zeta)y_1 = (1/2 - \zeta)|\Gamma| /k$ many edges for a vertex in the corresponding set $W_i$ of $\Gamma$. Thus $\Gamma$ has minimum degree 
    \[\delta(\Gamma) \ge (1-2/r+\mu/2 - 2\zeta) |\Gamma = (1-2/r+\mu/4) |\Gamma |\] 
    where $\zeta=\mu/8$ as we chose for Lemma~\ref{lem:K4free}.
    
    The important observation now is that every $K_r$ in $\Gamma$ corresponds to a multi-embedding of $K_r$ in $R$. This is an easy consequence of the construction of $\Gamma$. For every $K_r$ in $\Gamma$ take the mapping which maps to the vertex $i$ if the vertex of $K_r$ lies in the set $W_i$ in $\Gamma$. We never embed three vertices into a vertex of $R$ because all $W_i$'s are triangle free and if there are two clusters $W_i$ and $W_j$ into which we embed two vertices each, these vertices form a $K_4$ which means in $R$, $i$ and $j$ must be connected by a double-edge.
    By construction, the largest independent set of every cluster $W_i$ of $\Gamma$ is at most $\gamma |W_i|$ so $\alpha(\Gamma) \le \gamma |\Gamma|$. Then by the assumption of the lemma we have a fractional $K_r$-tiling $\mathcal{T}$ of $\Gamma$. We convert the fractional $K_r$-tiling of $\Gamma $ into a fractional $K_r$-tiling $\mathcal{T}'$ of $R$ by applying the mapping from $K_r$'s to $K_r$-multi-embeddings of $R$. So, for every multi-embedding $\mathcal{K}$ of a $K_r$ into $R$ we can define the set $L_{\mathcal{K}}$ to be the set of all $K_r$ in $\Gamma $ such that the multi-embedding $\mathcal{K}$ maps to the same partitions $V_i$ corresponding to $W_i$ in $\Gamma $. Then 
    
    \[ w_{\mathcal{T}'}(\mathcal{K}) \ge \sum_{K \in L_{\mathcal{K}} } \frac{ w_\mathcal{T}( K )}{y_1}  . \]
    
    So the total weight of $\mathcal{T}'$ must be at least $(1-\eta ) k$ in $R$. Then by Lemma~\ref{lem:frac_multi_to_almost} we can convert the fractional tiling into an almost cover of $G$ that covers at least $(1-2\eta)n$ vertices.
\end{proof}

For the proof of Lemma~\ref{lem:fracmat} we need the following lemma which gives us a stepwise improvement of any tiling we have as long as we do not cover a $(1-\eta)$ fraction of the vertices yet. Here a $\{K_r, K_{r+1}\}$-tiling is a disjoint union of $K_r$'s and $K_{r+1}$'s as subgraph. 

\begin{lemma}
	For every $r\in \NN$ and $\eta, \mu >0$ there exist $\rho$, $\gamma>0$ and $n_0\in \NN$ such that every graph $G$ on $n\ge n_0$ vertices with $\delta(G)\ge \left(1-\frac{2}{r}+\mu\right)n$ and $\alpha(G)< \gamma n$ has the following property: 
	
	Let $\mathcal{T}$ be a maximum $K_r$-tiling in $G$ with $|V(\mathcal{T})|\le (1-\eta)n$. Then there is a $\{K_r,K_{r+1}\}$-tiling which covers at least $|V(\mathcal{T})|+\rho n$ vertices.
	\label{lem:KrKr+1}
\end{lemma}

\begin{proof}
	Let $\mathcal{R}= V(G) \setminus V(\mathcal{T})$ be the set of all uncovered vertices in $G$. By Lemma~\ref{lem:RT} we know that the average degree inside $\mathcal{R}$ is less than $\left(1-2/(r-1) + \mu \right)|\mathcal{R}|$ as else there would be a $K_r$ inside $\mathcal{R}$ that we could add to $\mathcal{T}$ contradicting the maximality of $\mathcal{T}$. 
	We show that this implies that we can extend at least $\rho n$ of the $K_r$'s in $\mathcal{T}$ to $K_{r+1}$ where $\rho$ is some constant to be chosen later. 

	Let $T=V(\mathcal{T})$ and we are guaranteed that $|T|\ge \mu n$ as otherwise every vertex in $\mathcal{R}$ would have $\deg(v,\mathcal{R})\ge (1-2/r+\mu)n-\mu n>(1-2/(r-1))|\mathcal{R}|$ contradicting our upper bound on the average degree. Moreover, as every vertex in $\mathcal{R}$ has degree at least $(1-2/r+\mu)n$ and inside $\mathcal{R}$ we have an average degree less than $(1-2/(r-1))|\mathcal{R}|$ we know that the edges in between, $\deg(\mathcal{R}, T)$, are at least
	
	\[ \left(1-\frac{2}{r}+\mu\right)n|\mathcal{R}|-\left(1-\frac{2}{r-1}\right)|\mathcal{R}|^2= \left(1-\frac{2}{r}+\mu\right)|T||\mathcal{R}|+\frac{2}{r(r-1)}|\mathcal{R}|^2 \]
	
	edges to $T$. 
	Let $\mathcal{R}'\subseteq \mathcal{R}$ be the set of all vertices in $\mathcal{R}$ that have $\deg(v,T)> \left(1-\frac{2}{r}+\mu\right)|T|$ as we know that $\deg(v,T)\le |T|\le n$, we conclude that \[|\mathcal{R}'|\ge \frac{\frac{2}{r(r-1)}|\mathcal{R}|^2}{n}\ge \varphi |\mathcal{R}|\]
	for $\varphi = \frac{2}{r(r-1)}\eta$.
	
    \begin{figure}
    \centering
    \includegraphics[scale=0.8]{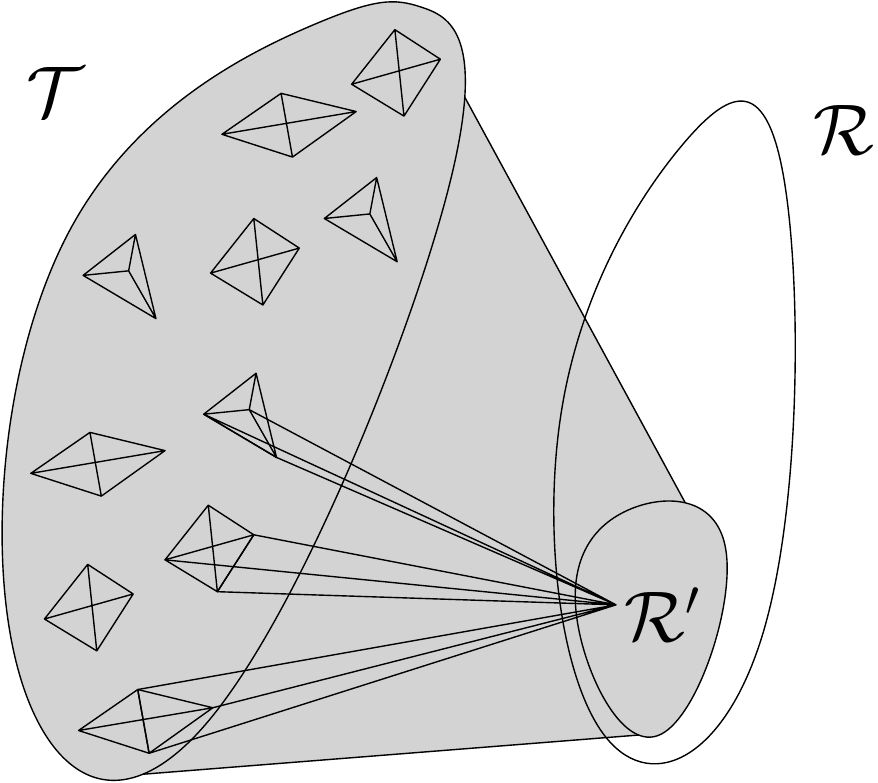} 
    \label{fig:greedyext}
    \caption{Greedy extending to $K_{r+1}$}
    \end{figure}
    
    We now use vertices or edges from $\mathcal{R}'$ to extend some $K_r$ from $\mathcal{T}$ to a $K_{r+1}$. Let $\mathcal{T'}$ be the set of all $K_{r}$ that we did not yet extend in this process and $\mathcal{R}'' \subseteq  \mathcal{R}'$ the set of unused vertices in $ \mathcal{R}'$ so far.
	The following claim asserts that the greedy process works.
	\begin{claim}
		If $\mathcal{R}'' \subseteq \mathcal{R}'$ is such that $(\mu/2r)|\mathcal{R}''|\ge \gamma n$ and for every vertex $v\in \mathcal{R}''$ we have ${T'} = V(\mathcal{T'})\subseteq T$ with $\deg(v, T') \ge \left(1-\frac{2}{r}+\frac{\mu}{2}\right)|T'|$, then we can find a $K_r$ in $\mathcal{T}'$ which can be extended to a $K_{r+1}$. 
		\label{claim:extendKr}
	\end{claim} 
	\begin{proof}
		If there is a $K_r$ in $\mathcal{T}'$ such that there is a vertex in $\mathcal{R}''$ which is connected to all vertices from this $K_r$, then we can extend it to a $K_{r+1}$. We can thus assume that every vertex in $\mathcal{R}''$ has at most $r-1$ edges to any $K_r$ in $\mathcal{T}'$. Then, the minimum degree condition implies that every vertex has at least $(\mu/2)|T'|$ copies of $K_r$ in $\mathcal{T}'$ such that $v$ is connected to exactly $r-1$ vertices of this $K_r$. 
		
		We can construct an auxiliary bipartite graph where the vertices in one partition are the copies of $K_r$ in $\mathcal{T}'$ and the other partition is formed by the vertices in $\mathcal{R}''$. Then the previous observation implies that this bipartite graph has at least $(\mu/2)|T'||\mathcal{R}''|$ edges and we can thus find a $K_r$ in $\mathcal{T'}$ such that at least $(\mu/2)|\mathcal{R}''|$ vertices from $\mathcal{R}''$ have exactly $r-1$ edges to this particular $K_r$. Call the set of these vertices $\mathcal{R}'''$ we can then further partition $\mathcal{R}'''$ into $\mathcal{R}'''_1,\ldots,\mathcal{R}'''_r$ where we put a vertex $v\in \mathcal{R}'''$ in $\mathcal{R}'''_i$ if and only if $v$ does not have an edge to the $i$th vertex in the $K_r$ (where the order of the vertices is arbitrary but fixed). Then there is some index $j$ such that $|\mathcal{R}'''_j|\ge (\mu/2r)|\mathcal{R}''|$. As we required that $\alpha(G)<\gamma n \le (\mu/2r)|\mathcal{R}''| \le |\mathcal{R}'''_j|$, there is an edge $e$ in $\mathcal{R}'''_j$. We can thus construct a $K_{r+1}$ by removing the $j$th vertex from the $K_r$ and adding the edge $e$ to the $K_r$. 
	\end{proof}

	Note that for every $K_{r+1}$ we construct we remove one $K_r$ from $\mathcal{T'}$ and at most two vertices from $\mathcal{R}''$. We choose $\rho$ maximal such that $\rho n\le (\mu / 2 r)|T|$ and $2 \rho n \le |\mathcal{R}'| - (2r / \mu)\gamma n$. 
	After the removal of at most $\rho n $ greedily formed $K_{r+1}$'s we are thus left with at least $|\mathcal{R}'|- 2 \rho n$ vertices in $\mathcal{R}''$ each of these vertices has $\deg(v,T') \ge (1-2/r+\mu/2)|T'|$. 
	Then Claim~\ref{claim:extendKr} gives that we can chose the $K_{r+1}$'s in a greedy manner until we extend $\rho n$ many $K_r$'s.
\end{proof}
Now we are ready to prove Lemma~\ref{lem:fracmat}. We restate the lemma for convenience of the reader.

\fracmat*
\begin{proof}[Proof of Lemma~\ref{lem:fracmat}]
	We start by taking a maximum $K_r$-tiling in $G$. If this covers more than $(1-\eta)n$ vertices, then we are done immediately. Else we repeatedly apply Lemma~\ref{lem:KrKr+1} while at every step blowing up each vertex of our graph $G$ with $r$ vertices. This follows the idea which emerged from \cite{treglown2016degree}. 
	After a constant number of blowups we can cover all but a $\eta^2$ fraction of the vertices with $K_r$'s. We then convert this tiling of the blown up graph into a fractional tiling of the original graph which misses at most $\eta^2 n$ of total weight, which directly implies that at most $\eta n$ vertices can have $w_{\mathcal{T}}(v)< 1-\eta $.

	In each of the steps we blow up the graph by a factor of $r$, that is we replace every vertex in the previous graph with a set of $r$ vertices and put complete bipartite graphs between all clusters that were connected by an edge in the previous graph. Note that this implies that for a $K_{r+1}$ in the previous graph we can find a perfect $K_r$ tiling in the blown up graph. 
	 We will repeat two steps:
	\begin{itemize}
		\item In the first step the \emph{enlargement step} here we start with a $K_r$ tiling which covers a $\lambda $ fraction of the vertices into a $\{K_r,K_{r+1}\}$ tiling that covers a $\lambda' >\lambda + \rho_{\ref{lem:KrKr+1}} (\eta^2, \mu, r)$ fraction
		\item The second step, the \emph{blow up step} blows up the graph and converts the given $\{K_r,K_{r+1}\}$-tiling into a $K_r$-tiling that covers a $\lambda'$ fraction.
	\end{itemize}
	Note that a $K_r$-tiling of any graph corresponding to a constant blow up by a factor of $s$ of $G$ which covers a $\lambda$ fraction of the vertices can be converted into a fractional $K_r$-tiling in $G$ with weight $\lambda n$. This can be done as follows. Let $\mathcal{T'}$ be a $K_r$-tiling in the blown up graph. We construct the fractional $K_r$-tiling $\mathcal{T}$ in $G$ in the following way. For every $K_r\in \mathcal{T'}$ by construction there is a copy of $K_r$ in $G$ which corresponds to this $K_r$ (in particular we cannot have two vertices which originate from the same vertex in $G$ as these vertices would come from an independent set). We add this $K_r$ to $\mathcal{T}$ with weight $1/s$. When there are multiple instances that correspond to the same $K_r$ in $G$ we just increase the weight by $1/s$ for each copy in the blown up graph. Let $G^s$ be the blowup of $G$ by factor of $s$, as the tiling we constructed covers $\lambda |G^s|$ vertices in $G^s$ we get that
	\[ \sum_{v\in V(G)}w_{\mathcal{T}}(v) = \sum_{K_r\in \mathcal{T'}}r \frac{1}{s}=\lambda |G^s| \frac{1}{s}=\lambda |G|.\]
	It thus suffices to show that for some number $s$, independent of $\gamma$ and $n$ we can find a $K_r$-tiling that covers $(1-\eta^2) |G^s|$ vertices in $G^s$. Let $\gamma
	=\gamma_{\ref{lem:KrKr+1}}(\eta^2,\mu,r)$ and $\rho=\rho_{\ref{lem:KrKr+1}}(\eta^2,\mu,r)$. Every time we apply Lemma~\ref{lem:KrKr+1} we newly cover a $\rho$ fraction of the vertices. We thus need to apply this lemma at most $1/\rho$ times. In each blow up step we replace one vertex from the previous graph by $r$ vertices. As we have to do at most $1/\rho$ blow up steps we know that $s\le r^{1/\rho}$. 
\end{proof}
Lemma~\ref{lem:almost} follows directly by applying Lemma~\ref{lem:fracmat} with $\mu_{\ref{lem:fracmat}}=\mu/4$ to $\Gamma$ from Lemma~\ref{lem:tilingtransferlemma} with $\mu$ and $\eta=\xi/4$. 

\section{Finishing the proof}

All that is left to do is to combine the results from the previous sections to prove the main theorem.

\begin{proof}[Proof of Theorem \ref{thm:mainresult}]

    Choose $\phi \le \mu/14r^2$ but independent from all other variables. Let $\xi=\xi_{\ref{lemma:absorbing}}$ where we apply Lemma \ref{lemma:absorbing} with $\phi$, $h=r$ and $t=6r+1$. Choose $\gamma$ small enough such that it satisfies Lemma \ref{lem:absorber_diamondpaths} as well as Lemma~\ref{lem:almost} dependent on the parameters $\mu$, $\phi$ and $\xi$. 
    
    In order to apply Lemma~\ref{lem:absorber_diamondpaths} to get a $\xi$-absorbing set, we show that for every choice of a $r$-vertex subset $S$ of $V(G)$ we can find $\phi n$ vertex disjoint $(S,6r+1)$-absorbers. We do this as follows. Start with an arbitrary $K_r$ that does not share any vertex with $S$ using Lemma~\ref{lem:RT}. Take an arbitrary bijection $g\colon V(K_r)\to S$ of the vertices of this $K_r$ to the vertices in $S$. Then use Lemma~\ref{lem:absorber_diamondpaths} to find disjoint $K_{r+1}$-diamond paths of length at most $7$ between each pair $(v,g(v))$ for all $v\in V(K_r)$. Add arbitrary $K_r$'s in case some paths where shorter until there are exactly $6r^2 + r$ vertices in total. We can repeat this $\phi n$ times without removing more than $(6r^2 + r)\phi n < (\mu/2)n$ vertices from the graph. Having these $\phi n$ many $(S,t)$-absorbers implies by Lemma \ref{lemma:absorbing} that there is some constant $\xi$ such that there is a $\xi$-absorbing set of size at most $\phi n$. Take such a set $A$ and put it aside.

    Note that as $|A|\le \phi n$ we know that for $G'= G\setminus A$ we have $\delta(G')\ge (1-2/r+\mu')n'$ and $\alpha(G') \le \gamma' n' $ where $n'=|V(G')|$, $\mu'=\mu/2$ and $\gamma' = 2\gamma$. We then apply Lemma~\ref{lem:almost} with $\xi$ from Lemma~\ref{lemma:absorbing} to $G'$ to get a tiling that covers all but at most $\xi n'$ vertices.
    Let $V_R$ be the set of vertices that remain uncovered in Lemma~\ref{lem:almost}. By construction we have $|V_R|\le \xi n'\le \xi n$ and thus we can use the absorber $A$ to cover $A\cup V_R$.
\end{proof}
\section{Final Remarks}

In Section~\ref{sec:absorbers} we prove the existence of an absorber as in Definition~\ref{def:absorber} and in Section~\ref{sec:almost} we find an almost cover by Lemma~\ref{lem:almost}. These together imply Theorem \ref{thm:mainresult}. Note that for $r < 4$ the same properties hold in spirit, but we have a different problem with divisibility of the connected components as described in \cite{balogh2016triangle}.

We showed that additionally having a low independence number in a graph significantly improves the statement from the famous Hajnal-Szemer\'edi theorem. We can take twice the size of a clique and this shows that large independent sets are really the bottleneck of the theorem.
The methods we apply might work for embedding any kind of structure into graphs with low independence number. In particular, since sparse random graphs have this property, as an immediate corollary we get that adding a dense graph (with $\delta(G) \ge (1-2/r + \mu)n$) on top of a sparse random graph ($G_{n,p}$ with $p>C\log n / n$) in any adversarial way we can still find a $K_r$-factor with high probability.

Still, we have that the bottleneck seems to be that there is a large triangle free set of size $2n / r+1$. From this set we can take at most two vertices per clique so it is impossible to cover with $n/r$ cliques of size $r$.
A natural question would be whether this extends to the natural generalization of the independence number. That is, instead of at least an edge in any subgraph of size $\gamma n$ we even find a triangle. More formally, let $\alpha_\ell(G)$ be the size of the largest set of vertices in $G$ that does not contain an induced $K_\ell$. 
The following question is a generalization of Theorem~\ref{thm:mainresult}. 
\begin{question}\label{q:generalalpha}
Is it true that for every $\ell,r\in \mathbb{N}$ with $\ell\le r$ and $\mu>0$ there is a constant $\gamma$ and $n_0\in \NN$ such that every graph on $n\ge n_0$ vertices where $r$ divides $n$, with $\delta(G)\ge\max\{1/2+\mu,(1-\ell/r+\mu)n\}$ and $\alpha_\ell(G)<\gamma n$ has a $K_r$-factor.
\end{question}

In our results, the independence number is always dependent on the $\mu$ of the minimum degree. Even though the examples that we know of only require $\alpha(G)$ to be smaller than $n/r$. Does there exist a fixed constant $\gamma$ dependent only on $r$ such that all graphs with minimum degree $\delta(G) \ge (1-2/r+ \mu)n $ have a $K_r$-factor? Combined with Question \ref{q:generalalpha} we ask whether the following is true.

\begin{question}
\label{q:generalgamma}
Is there a $\gamma > 0$ dependent only on $r$ and $\ell$ such that for every $\mu>0$ there is an $n_0$ large enough such that every graph on $n\ge n_0$ vertices where $r$ divides $n$ with $\delta(G)\ge\max\{1/2+\mu,(1-\ell/r+\mu)n\}$ and $\alpha_\ell(G)<\gamma n$ has a $K_r$-factor.\textbf{}
\end{question}

As a first step it would be interesting to answer Question \ref{q:generalgamma} for $\ell=2$ and $r=4$.
Clearly, we cannot hope for $\gamma\ge 1/r$ but it would be interesting to see how far we can push $\gamma$ towards this bound. 

\section*{Acknowledgment}
We want to thank Rajko Nenadov for bringing the problem to our attention and helpful discussions when we started the project. We thank the anonymous referee for the useful comments and suggestions. 
\pagebreak

\bibliographystyle{abbrv}
\bibliography{factors}

\end{document}